\documentclass[a4paper,12pt]{article}

\usepackage[utf8]{inputenc}
\usepackage[english,russian]{babel}
\usepackage{amsmath, amssymb, amsthm}
\usepackage{enumerate}
\usepackage[colorlinks=true,linkcolor=blue,citecolor=blue]{hyperref}

\usepackage{tikz}
\usepackage{mdframed}
\usepackage{pgfplots}
\pgfplotsset{width=7cm,compat=1.10}
\pgfplotsset{
  /pgfplots/colormap={pink}{%
    color(0cm) = (blue);
    color(1cm) = (cyan!50!blue);
    color(2cm) = (cyan!50);
    color(3cm) = (cyan) }
}
\usetikzlibrary{calc,arrows,decorations.pathreplacing,fadings,3d,positioning}

\title{О равномерных диофантовых экспонентах решёток
       \thanks{Работа поддержана Фондом развития теоретической физики и математики <<БАЗИС>>}}
\author{О.\,Н.\,Герман}
\date{}

\theoremstyle{definition}
\newtheorem{definition}{Определение}

\newtheorem*{notation*}{Обозначение}

\theoremstyle{remark}

\newtheorem*{remark*}{Замечание}

\theoremstyle{plain}
\newtheorem{theorem}{Теорема}
\newtheorem*{theorem*}{Теорема}
\newtheorem{lemma}{Лемма}
\newtheorem*{lemma*}{Лемма}

\newtheorem*{corollary*}{Следствие}

\DeclareMathOperator{\conv}{conv}

\renewcommand{\phi}{\varphi}

\renewcommand{\vec}[1]{\mathbf{#1}}
\renewcommand{\geq}{\geqslant}
\renewcommand{\leq}{\leqslant}

\newcommand{\e}{\varepsilon}

\newcommand{\R}{\mathbb{R}}
\newcommand{\Z}{\mathbb{Z}}
\newcommand{\Q}{\mathbb{Q}}
\newcommand{\N}{\mathbb{N}}

\newcommand{\La}{\Lambda}

\newcommand{\cB}{\mathcal{B}}

\newcommand{\cH}{\mathcal{H}}

\newcommand{\cK}{\mathcal{K}}

\newcommand{\cP}{\mathcal{P}}
\newcommand{\cQ}{\mathcal{Q}}

\begin{document}

\maketitle

\begin{abstract}
  В данной работе мы изучаем спектр значений слабых равномерных диофантовых экспонент решёток и получаем полное его описание в двумерном случае.
\end{abstract}

\section{Введение}

\subsection{Диофантовы экспоненты векторов}

В 1842 году в работе \cite{dirichlet} Дирихле опубликовал свою знаменитую теорему, положившую начало теории диофантовых приближений. Существует два классических способа её формулировки, которые приводят к понятиям \emph{регулярной} и \emph{равномерной} диофантовых экспонент. Например, в задаче о совместных приближениях для набора $\Theta=(\theta_1,\ldots,\theta_n)\in\R^n$ равномерный вариант теоремы Дирихле говорит, что система 
\begin{equation} \label{eq:dirichlet_simultaneous}
\begin{cases}
  0<|x|\leq t \\
  \max_{1\leq i\leq n}|\theta_ix-y_i|\leq t^{-\gamma}
\end{cases}
\end{equation}
имеет ненулевое целочисленное решение (относительно $x,y_1,\ldots,y_n$) при $\gamma=1/n$ и произвольном $t\geq1$.

\begin{definition} \label{def:belpha_simultaneous}
  Супремум вещественных чисел $\gamma$, для которых существует сколь угодно большое $t$, такое что (соотв. для которых при любом достаточно большом $t$) система неравенств \eqref{eq:dirichlet_simultaneous} имеет ненулевое решение $(x,y_1,\ldots,y_n)\in\Z^{n+1}$, называется \emph{регулярной} (соотв. \emph{равномерной}) \emph{диофантовой экспонентой} вектора $\Theta$ и обозначается $\omega(\Theta)$ (соотв. $\hat\omega(\Theta)$).
\end{definition}

Из теоремы Дирихле мгновенно следуют <<тривиальные>> оценки
\[
  \omega(\Theta)\geq\hat\omega(\Theta)\geq 1/n.
\]
Справедливо также неравенство $\hat\omega(\Theta)\leq1$ для каждого вектора $\Theta$, имеющего хотя бы одну иррациональную компоненту. Причина --- в том же двумерном эффекте, из-за которого в случае $n=1$ равномерная диофантова экспонента становится тривиальной. Подробности можно найти, например, в обзоре \cite{german_UMN_2023}.

Параметрическая геометрия чисел (см. \cite{schmidt_summerer_2009, schmidt_summerer_2013, roy_annals_2015, roy_zeitschrift_2016}) позволяет доказать, что спектр возможных значений каждой из экспонент $\omega(\Theta)$ и $\hat\omega(\Theta)$ описывается упомянутыми выше неравенствами. Совсем не так обстоит дело с диофантовыми экспонентами решёток.

\subsection{Диофантовы экспоненты решёток}

Для каждого $\vec z=(z_1,\ldots,z_d)\in\R^d$ положим
\[
  |\vec z|=\max_{1\leq i\leq d}|z_i|,
  \qquad
  \Pi(\vec z)=\prod_{\begin{subarray}{c}1\leq i\leq d\end{subarray}}|z_i|^{1/d}.
\]
Определим также для каждого $d$-набора $\pmb\lambda=(\lambda_1,\ldots,\lambda_d)\in\R_+^d$ параллелепипед 
\begin{equation}\label{eq:prallelepipeds_lattice_exp}
  \cP(\pmb\lambda)=\Big\{\,\vec z=(z_1,\ldots,z_d)\in\R^d \ \Big|\ |z_i|\leq\lambda_i,\ i=1,\ldots,d \Big\}.
\end{equation}

\begin{definition} \label{def:regular_lattice_exponents}
  Пусть $\La$ --- решётка в $\R^d$ полного ранга. Супремум вещественных чисел $\gamma$, для которых существует сколь угодно большое $t$ и $d$-набор $\pmb\lambda\in\R_+^d$, удовлетворяющий неравенствам
  \[
    |\pmb\lambda|\leq t,
    \qquad
    \Pi(\pmb\lambda)\leq t^{-\gamma},
  \]
  такой что параллелепипед $\cP(\pmb\lambda)$ содержит ненулевые точки решётки $\La$, называется \emph{(регулярной) диофантовой экспонентой} решётки $\La$ и обозначается $\omega(\La)$.
\end{definition}

В отличие от задачи о совместных приближениях, равномерный аналог диофантовой экспоненты решётки можно определять по-разному, ибо в определении \ref{def:regular_lattice_exponents} квантор существования стоит как при $t$, так и при $\pmb\lambda$.

\begin{definition} \label{def:strong_uniform_lattice_exponents}
  Пусть $\La$ --- решётка в $\R^d$ полного ранга. Супремум вещественных чисел $\gamma$, таких что для любого достаточно большого $t$ и любого $d$-набора $\pmb\lambda\in\R_+^d$, удовлетворяющего равенствам
  \[
    |\pmb\lambda|=t,
    \qquad
    \Pi(\pmb\lambda)=t^{-\gamma},
  \]
  параллелепипед $\cP(\pmb\lambda)$ содержит ненулевые точки решётки $\La$, называется \emph{сильной равномерной диофантовой экспонентой} решётки $\La$ и обозначается $\hat\omega(\La)$.
\end{definition}

\begin{definition} \label{def:weak_uniform_lattice_exponents}
  Пусть $\La$ --- решётка в $\R^d$ полного ранга. Супремум вещественных чисел $\gamma$, таких что для любого достаточно большого $t$ существует $d$-набор $\pmb\lambda\in\R_+^d$, удовлетворяющий неравенствам
  \[
    |\pmb\lambda|\leq t,
    \qquad
    \Pi(\pmb\lambda)\leq t^{-\gamma},
  \]
  такой что параллелепипед $\cP(\pmb\lambda)$ содержит ненулевые точки решётки $\La$, называется \emph{слабой равномерной диофантовой экспонентой} решётки $\La$ и обозначается $\hat{\hat\omega}(\La)$.
\end{definition}

Из теоремы Минковского о выпуклом теле и определений диофантовых экспонент решёток легко следуют неравенства
\[
  \omega(\La)\geq
  \hat{\hat\omega}(\La)\geq
  \hat\omega(\La)\geq0.
\]
В работе \cite{german_comm_math_2023} показано, что сильные равномерные диофантовы экспоненты решёток принимают лишь тривиальные значения, то есть $0$ и $\infty$. Цель же настоящей работы --- показать, что спектр значений слабой равномерной диофантовой экспоненты решёток нетривиален.

Статья организована следующим образом. В параграфе \ref{sec:spectra} мы обсуждаем, что известно про спектры диофантовых экспонент решёток и формулируем основной результат статьи. В параграфе \ref{sec:continued_fractions_and_lattices} мы обсуждаем связь между диофантовыми экспонентами вещественных чисел и диофантовыми экспонентами решёток в размерности $2$. В параграфе \ref{sec:relative_and_hyperbolic_minima} мы напоминаем определение относительных минимумов решётки и вводим понятие гиперболического минимума решётки. Наконец, в параграфе \ref{sec:the_proof} мы приводим доказательство основного результата данной статьи.

\section{Спектры диофантовых экспонент решёток}\label{sec:spectra}

В работе \cite{german_lattice_exponents_spectrum} показано, что при $d=2$ спектр значений экспоненты $\omega(\La)$ совпадает с лучом $[0,+\infty]$ и что при $d\geq3$ он содержит отрезок
\[
  \bigg[3-\frac{d}{(d-1)^2}\,,\,+\infty\bigg].
\]
Конечно же, $0$ также принадлежит спектру при любом $d$, ибо $\omega(\La)=0$ всякий раз, когда функция $\Pi(\vec z)$ отделена от нуля в ненулевых точках решётки $\La$. К примеру, это так для любой решётки полного модуля вполне вещественного алгебраического расширения поля $\Q$ (см. \cite{borevich_shafarevich}). Теорема Шмидта о подпространствах позволяет строить решётки с регулярными диофантовыми экспонентами, принимающими значения
\begin{equation*}
  \frac{\,ab\,}{cd}\,,\qquad
  \begin{array}{l}
    a,b,c\in\N, \\
    a+b+c=d.
  \end{array}
\end{equation*}
Описание соответствующих решёток можно найти в работе \cite{german_lattice_transference}. Естественно ожидать, что при любом $d$ спектр значений экспоненты $\omega(\La)$ совпадает с лучом $[0,+\infty]$. Однако же при $d\geq3$ это до сих пор не доказано.

Как было сказано выше, спектр значений сильных равномерных диофантовых экспонент решёток тривиален. А именно, в работе \cite{german_comm_math_2023} показано, что если решётка $\La$ подобна подрешётке $\Z^d$ по модулю действия группы невырожденных диагональных операторов, то $\hat\omega(\La)=\infty$, а если нет, то $\hat\omega(\La)=0$.

Данная работа посвящена описанию спектра значений слабой равномерной диофантовой экспоненты решёток в случае $d=2$. Следующее утверждение является основным результатом статьи.

\begin{theorem}\label{t:spectrum_of_the_weak}
  При $d=2$ спектр значений $\hat{\hat\omega}(\Lambda)$ совпадает с множеством $[0,\infty]$.
\end{theorem}

\section{Цепные дроби и решётки в $\R^2$}\label{sec:continued_fractions_and_lattices}

\subsection{Диофантова экспонента вещественного числа}

При $n=1$ диофантовы экспоненты векторов становятся диофантовыми экспонентами вещественных чисел. Для каждого $\theta\in\R$ в соответствии с определением \ref{def:belpha_simultaneous} экспонента $\omega(\theta)$ равна супремуму таких вещественных $\gamma$, что неравенство
\[
  |\theta x-y|\leqslant|x|^{-\gamma}
\]
имеет бесконечно много решений в ненулевых целых $x$, $y$.

Благодаря классическому соотношению
\begin{equation}\label{eq:omega_of_a_number_vs_its_continued_fraction}
  \omega(\theta)=1+\limsup_{n\to+\infty}\frac{\ln a_{n+1}}{\ln q_n},
\end{equation}
связывающему $\omega(\theta)$ с цепной дробью числа $\theta$,
\[
  \theta=[a_0;a_1,a_2,\ldots],\quad
  \frac{p_n}{q_n}=[a_0;a_1,\ldots,a_n],
\]
можно явно строить числа с любой заданной наперёд диофантовой экспонентой $\omega(\theta)\geq1$.

\subsection{Диофантовы экспоненты вещественных чисел и регулярные диофантовы экспоненты решёток}

Покажем, как связаны диофантовы экспоненты решёток с диофантовыми экспонентами вещественных чисел.

Для любой решётки $\La$ полного ранга в соответствии с определением \ref{def:regular_lattice_exponents} экспонента $\omega(\La)$ равна супремуму таких вещественных $\gamma$, что неравенство
\[
  \Pi(\vec z)\leq|\vec z|^{-\gamma}
\]
имеет бесконечно много решений в $\vec z\in\La$.

Пусть $\theta$, $\eta$ --- различные вещественные числа, $\theta>1$, $\eta>1$. Рассмотрим линейные формы $L_1$, $L_2$ от двух переменных с коэффициентами, записанными в строчках матрицы
\[
  A=
  \begin{pmatrix}
    \theta & -1 \\
    1 & \phantom{-}\eta
  \end{pmatrix},
\]
и положим
\[
  \La=A\Z^2=
  \Big\{\big(L_1(\vec u),L_2(\vec u)\big)\,\Big|\,\vec u\in\Z^2 \Big\}.
\]
Тогда для каждого $\,\vec z=\big(L_1(\vec u),L_2(\vec u)\big)\in\La$, где $\ \vec u=(x,y)\in\Z^2$, справедливо
\[
  \Pi(\vec z)^2=|L_1(\vec u)|\cdot|L_2(\vec u)|=|\theta x-y|\cdot|x+\eta y|,
\]
а также
\[
  |\vec z|\asymp|L_i(\vec u)|\asymp|\vec u|\asymp|x|\asymp|y|
  \ \text{ при }\ |L_j(\vec u)|\leq1,\ i\neq j,
\]
где константы, подразумеваемые символом ``$\asymp$'', зависят лишь от $\theta$ и $\eta$. Следовательно, при $|x+\eta y|\leq1$ справедливо
\[
  \big(\Pi(\vec z)\cdot|\vec z|^\gamma\big)^2
  \asymp
  |\theta x-y|\cdot|x|^{1+2\gamma},
\]
тогда как при $|\theta x-y|\leq1$ справедливо
\[
  \big(\Pi(\vec z)\cdot|\vec z|^\gamma\big)^2
  \asymp
  |x+\eta y|\cdot|y|^{1+2\gamma}.
\]
Отсюда видим, что 
\begin{equation}\label{eq:omega_La_vs_the_maximum}
  \omega(\La)=\frac12\Big(\max\big(\omega(\theta),\omega(\eta)\big)-1\Big).
\end{equation}

При $\omega(\eta)=\omega(\theta)$ пропадает необходимость брать максимум. Особенно удобным для некоторых приложений является выбор $\eta=\theta$. Тогда для решётки
\begin{equation}\label{eq:La_theta}
  \La=
  \begin{pmatrix}
    \theta & -1 \\
    1 & \phantom{-}\theta
  \end{pmatrix}
  \Z^2
\end{equation}
соотношение \eqref{eq:omega_La_vs_the_maximum} сводится к равенству
\[
  \omega(\La)=
  \frac{\omega(\theta)-1}2\,,
\]
из которого очевидным образом следует, что спектр значений экспоненты $\omega(\La)$ совпадает с отрезком $[0,\infty]$.

Однако для описания спектра значений $\hat{\hat\omega}(\La)$ решёток вида \eqref{eq:La_theta} недостаточно. Как увидим ниже, полезно рассматривать пары различных чисел $\theta$, $\eta$, диофантовы свойства которых связаны определённым образом друг с другом.

\section{Относительные и гиперболические минимумы}\label{sec:relative_and_hyperbolic_minima}

\begin{definition}
  Пусть $\La$ --- решётка полного ранга в $\R^d$. Точка $\vec z=(z_1,\ldots,z_d)\in\La$ называется \emph{относительным минимумом} решётки $\La$, если не существует таких ненулевых точек $\vec w=(w_1,\ldots,w_d)\in\La$, что
  \[
    |w_i|\leq|z_i|,\quad i=1,\ldots,d,
    \qquad\text{ и }\qquad
    \sum_{i=1}^d|w_i|<\sum_{i=1}^d|z_i|.
  \] 
\end{definition}

Иными словами, точка $\vec z\in\La$ является относительным минимумом решётки $\La$, если в параллелепипеде $\cP(\vec z)$, помимо вершин и точки начала координат, точек решётки нет.

\begin{definition}\label{def:hyperbolic_minimum}
  Пусть $\La$ --- решётка полного ранга в $\R^d$. Будем называть точку $\vec z\in\La$ \emph{гиперболическим минимумом} решётки $\La$, если не существует таких ненулевых точек $\vec w\in\La$, что 
  \[
    |\vec w|\leq|\vec z|
    \qquad\text{ и }\qquad
    \Pi(\vec w)<\Pi(\vec z).
  \] 
\end{definition}

Положим
\[
  \cH(\vec z)=\Big\{\vec w\in\R^d \,\Big|\, |\vec w|\leq|\vec z|,\ \Pi(\vec w)<\Pi(\vec z) \Big\}.
\]
Из определения \ref{def:hyperbolic_minimum} очевидно, что ненулевая точка $\vec z$ решётки $\La$ является \emph{гиперболическим минимумом} этой решётки тогда и только тогда, когда $\cH(\vec z)\cap\La=\{\vec 0\}$.

\begin{lemma}\label{l:hyperbolic_is_also_relative}
  Каждый гиперболический минимум решётки $\La$ является и относительным минимумом этой решётки.
\end{lemma}

\begin{proof}
  Если $\vec z$ --- гиперболический минимум решётки $\La$, ненулевые точки этой решётки, содержащиеся в $\cH(\vec z)$, обязаны располагаться на пересечении границы $\cH(\vec z)$ с поверхностью $\big\{ \vec w\in\R^d \,\big|\, \Pi(\vec w)=\Pi(\vec z) \big\}$. В частности, в этом случае любая ненулевая точка решётки $\La$, содержащаяся в $\cP(\vec z)$, обязана быть вершиной $\cP(\vec z)$. Таким образом, каждый гиперболический минимум решётки $\La$, действительно, является и относительным минимумом этой решётки.
\end{proof}

\subsection{Последовательные гиперболические минимумы}

Если $\vec z$ --- гиперболический минимум решётки $\La$, то, конечно же, точка $-\vec z$ также является гиперболическим минимумом этой решётки. Может так оказаться, что существуют другие гиперболические минимумы с такими же значениями функционалов $|\cdot|$ и $\Pi(\,\cdot\,)$. Выберем из каждого такого набора по одному представителю и упорядочим их по возрастанию $|\cdot|$. Получим последовательность $\vec z_1,\vec z_2,\vec z_3,\ldots$ гиперболических минимумов (см. рис. \ref{fig:successive_hyperbolic_minima}), причём $|\vec z_k|<|\vec z_{k+1}|$ и $\hat\cH(\vec z_k)\cap\La=\{\vec 0\}$, где
\[
  \hat\cH(\vec z_k)=\Big\{\vec w\in\R^d \,\Big|\, |\vec w|<|\vec z_{k+1}|,\ \Pi(\vec w)<\Pi(\vec z_k) \Big\}.
\]

\begin{figure}[h]
\centering
\begin{tikzpicture}[domain=-5.1:5.1,scale=1.3]

  \draw (-5.3,0) -- (5.3,0); 
  \draw (0,-5.3) -- (0,5.3); 

  \fill[fill=blue,opacity=0.15]
      plot [domain=0.2:5] (\x,{1/(\x)}) --
      plot [domain=5:0.2] (\x, {-1/(\x)}) --
      plot [domain=-0.2:-5] (\x,{1/(\x)}) --
      plot [domain=-5:-0.2] (\x, {-1/(\x)}) --
      cycle;

  \fill[fill=red,opacity=0.15]
      plot [domain=0.75:2] (\x,{1.5/(\x)}) --
      plot [domain=2:0.75] (\x, {-1.5/(\x)}) --
      plot [domain=-0.75:-2] (\x,{1.5/(\x)}) --
      plot [domain=-2:-0.75] (\x, {-1.5/(\x)}) --
      cycle;

  \draw[color=blue]
      plot [domain=0.2:5] (\x,{1/(\x)}) --
      plot [domain=5:0.2] (\x, {-1/(\x)}) --
      plot [domain=-0.2:-5] (\x,{1/(\x)}) --
      plot [domain=-5:-0.2] (\x, {-1/(\x)}) --
      cycle;

  \draw[color=red]
      plot [domain=0.75:2] (\x,{1.5/(\x)}) --
      plot [domain=2:0.75] (\x, {-1.5/(\x)}) --
      plot [domain=-0.75:-2] (\x,{1.5/(\x)}) --
      plot [domain=-2:-0.75] (\x, {-1.5/(\x)}) --
      cycle;

  \node[fill=black,circle,inner sep=1.3pt] at (0,0) {};
  \node[fill=black,circle,inner sep=1.3pt] at (1.03,1.5/1.03) {};
  \node[fill=black,circle,inner sep=1.3pt] at (-1.03,-1.5/1.03) {};
  \node[fill=black,circle,inner sep=1.3pt] at (-0.5,2) {};
  \node[fill=black,circle,inner sep=1.3pt] at (0.5,-2) {};
  \node[fill=black,circle,inner sep=1.3pt] at (5,0.1) {};
  \node[fill=black,circle,inner sep=1.3pt] at (-5,-0.1) {};
  
  \draw(1.1,1.5/1.1+0.1) node[right]{$\vec z_{k-1}$};
  \draw(-1.03,-1.5/1.1-0.1) node[left]{$-\vec z_{k-1}$};
  \draw(0.47,-2.23) node[right]{$\vec z_k$};
  \draw(-0.47,2.23) node[left]{$-\vec z_k$};
  \draw(5,0.2) node[right]{$\vec z_{k+1}$};
  \draw(-4.93,-0.2) node[left]{$-\vec z_{k+1}$};

  \draw(4,0.3) node[above]{${\color[rgb]{0,0,1}\hat\cH(\vec z_k)}$};
  \draw(1.3,-1.35) node[right]{${\color[rgb]{1,0,0}\hat\cH(\vec z_{k-1})}$};

\end{tikzpicture}
\caption{Последовательные гиперболические минимумы}\label{fig:successive_hyperbolic_minima}
\end{figure}

Если функционал $\Pi(\,\cdot\,)$ принимает сколь угодно малые значения на ненулевых точках решётки $\La$, но при этом ни на какой такой точке не обращается в ноль, гиперболических минимумов бесконечно много. В этом случае для экспонент $\omega(\La)$, $\hat{\hat\omega}(\La)$ справедливы равенства
\[
  \omega(\La)=\sup\Big\{ \gamma\in\R \,\Big|\, \forall\,K\in\N\,\ \exists\,k\geq K:\,\Pi(\vec z_k)\leq|\vec z_k|^{-\gamma} \Big\},\ \
\]
\begin{equation}\label{eq:hat_hat_omega_in_terms_of_hyperbolic_minima}
  \hat{\hat\omega}(\La)=\sup\Big\{ \gamma\in\R \,\Big|\, \exists\,K\in\N:\,\forall k\geq K\ \Pi(\vec z_k)\leq|\vec z_{k+1}|^{-\gamma} \Big\}.
\end{equation}

\subsection{Относительные минимумы в двумерном случае}\label{subsec:relative_minima_in_dimension_2}

Пусть далее $d=2$. Тогда, как показано, например, в работе \cite{german_mathnotes_2006}, множество относительных минимумов решётки $\La=A\Z^2$, где
\[
  A=
  \begin{pmatrix}
    \theta & -1 \\
    1 & \phantom{-}\eta
  \end{pmatrix},
  \qquad
  \theta>1,
  \ \
  \eta>1,
\]
совпадает с множеством вершин её четырёх полигонов Клейна
\[
  \cK_{\pmb\e}(\La)=
  \conv\Big(\Big\{ \vec z=(z_1,z_2)\in\La\backslash\{\vec 0\} \,\Big|\, \e_iz_i\geq0,\ i=1,2 \Big\}\Big),
\]
\[
  \pmb\e=(\e_1,\e_2),
  \quad
  \e_1,\e_2=\pm1.
\]
Вершины же этих полигонов Клейна соответствуют подходящим дробям чисел $\theta$, $\eta$. А именно, их вершины являются образами при действии оператора $A$ точек
\begin{equation}\label{eq:vertices}
  \pm(1,0),\quad
  \pm(0,1),\quad
  \pm(q_k,p_k),\quad
  \pm(r_k,-s_k),\quad
  k=0,1,2,\ldots,
\end{equation}
где $p_k/q_k$ и $r_k/s_k$ --- подходящие дроби чисел $\theta$ и $\eta$ соответственно. Подробное описание данного соответствия можно найти в работе \cite{german_tlyustangelov_MJCNT_2016}, а также в обзоре \cite{german_UMN_2023}.

В частности, ввиду леммы \ref{l:hyperbolic_is_also_relative} 
все гиперболические минимумы решётки $\La$ являются образами при действии оператора $A$ каких-то точек вида \eqref{eq:vertices}.

\section{Доказательство теоремы \ref{t:spectrum_of_the_weak}}\label{sec:the_proof}

Зафиксируем $\beta>1$. Построим числа
\[
  \theta=[a_0;a_1,a_2,\ldots],\qquad
  \eta=[b_0;b_1,b_2,\ldots]
\]
следующим образом. Положим $a_0=b_0=1$, $a_1=2$, $b_1=[2^\beta]+1$ и для каждого целого $k\geq2$ положим
\begin{equation}\label{eq:next_partial_qoutients_definition}
  a_k=\bigg[\frac{s_{k-1}^\beta-q_{k-2}}{q_{k-1}}\bigg]+1,
  \qquad
  b_k=\bigg[\frac{q_k^\beta-s_{k-2}}{s_{k-1}}\bigg]+1,
\end{equation}
где $[\,\cdot\,]$ обозначает целую часть, $q_k$ --- знаменатель дроби $p_k/q_k=[a_0;a_1,a_2,\ldots,a_k]$, $s_k$ --- знаменатель дроби $r_k/s_k=[b_0;b_1,b_2,\ldots,b_k]$.

\paragraph{Оценки роста знаменателей и неполных частных.}

Ввиду рекуррентных соотношений
\[
  q_k=a_kq_{k-1}+q_{k-2},
  \qquad
  s_k=b_ks_{k-1}+s_{k-2}
\]
из \eqref{eq:next_partial_qoutients_definition} следует, что для всех натуральных $k$ справедливы неравенства
\begin{equation}\label{eq:consecutive_denominators_inequalities}
\begin{array}{ll}
  q_k>
  \dfrac{s_{k-1}^\beta-q_{k-2}}{q_{k-1}}\cdot q_{k-1}+q_{k-2}=
  s_{k-1}^\beta,
  &
  s_k>
  \dfrac{q_k^\beta-s_{k-2}}{s_{k-1}}\cdot s_{k-1}+s_{k-2}=
  q_k^\beta,
  \\ \vphantom{1^{\Big|}}
  q_k<
  s_{k-1}^\beta+q_{k-1}<
  s_{k-1}^\beta+s_{k-1}^{1/\beta}<
  2s_{k-1}^\beta,
  &
  s_k<
  q_k^\beta+s_{k-1}<
  q_k^\beta+q_k^{1/\beta}<
  2q_k^\beta,
\end{array}
\end{equation}
а для всех достаточно больших $k$
\begin{equation}\label{eq:partial_quotients_vs_denominators}
\begin{array}{l}
  a_k>
  \dfrac{s_{k-1}^\beta-q_{k-2}}{q_{k-1}}>
  \dfrac{q_{k-1}^{\beta^2}-q_{k-1}^{1/\beta^2}}{q_{k-1}}>
  \dfrac12\,q_{k-1}^{\beta^2-1},
  \\ \vphantom{1^{\bigg|}}
  a_k<
  \dfrac{s_{k-1}^\beta-q_{k-2}}{q_{k-1}}+1<
  \dfrac{2^\beta q_{k-1}^{\beta^2}}{q_{k-1}}+1=
  2^\beta q_{k-1}^{\beta^2-1}+1,
  \\ \vphantom{1^{\bigg|}}
  b_k>
  \dfrac{q_k^\beta-s_{k-2}}{s_{k-1}}>
  \dfrac{s_{k-1}^{\beta^2}-s_{k-1}^{1/\beta^2}}{s_{k-1}}>
  \dfrac12\,s_{k-1}^{\beta^2-1},
  \\ \vphantom{1^{\bigg|}}
  b_k<
  \dfrac{q_k^\beta-s_{k-2}}{s_{k-1}}+1<
  \dfrac{2^\beta s_{k-1}^{\beta^2}}{s_{k-1}}+1=
  2^\beta s_{k-1}^{\beta^2-1}+1.
\end{array}
\end{equation}
В частности,
\begin{equation}\label{eq:denominators_asymptotics}
  q_k\asymp s_{k-1}^\beta,
  \qquad
  s_k\asymp q_k^\beta
\end{equation}
и
\begin{equation}\label{eq:partial_quotients_asymptotics}
  \,\quad
  a_{k+1}\asymp
  q_k^{\beta^2-1},
  \qquad
  b_{k+1}\asymp
  s_k^{\beta^2-1}.
\end{equation}

\paragraph{Неравенства для произведений.}

Как известно (см., например, \cite{lang,khintchine_CF,schmidt_DA}),
\[
  \frac1{a_{k+1}+2}<
  q_k|q_k\theta-p_k|<
  \frac1{a_{k+1}}\,,
  \qquad
  \frac1{b_{k+1}+2}<
  s_k|s_k\eta-r_k|<
  \frac1{b_{k+1}}\,.
  \qquad
\]
Учитывая \eqref{eq:partial_quotients_vs_denominators}, получаем, что для всех достаточно больших $k$
\begin{equation}\label{eq:bounds_for_the_products}
  \frac1{2^\beta q_k^{\beta^2-1}+3}<
  q_k|q_k\theta-p_k|<
  \frac2{q_k^{\beta^2-1}}\,,
  \qquad
  \frac1{2^\beta s_k^{\beta^2-1}+3}<
  s_k|s_k\eta-r_k|<
  \frac2{s_k^{\beta^2-1}}\,.
\end{equation}
Из \eqref{eq:consecutive_denominators_inequalities} и \eqref{eq:bounds_for_the_products} следует, что при фиксированном $\beta$ для всех достаточно больших $k$ справедливо
\[
  q_k|q_k\theta-p_k|>
  \frac1{2^\beta q_k^{\beta^2-1}+3}>
  \frac6{q_k^{\beta(\beta^2-1)}}>
  \frac6{s_k^{\beta^2-1}}>
  3s_k|s_k\eta-r_k|,
  \qquad\,\
\]
\[
  s_k|s_k\eta-r_k|>
  \frac1{2^\beta s_k^{\beta^2-1}+3}>
  \frac6{s_k^{\beta(\beta^2-1)}}>
  \frac6{q_{k+1}^{\beta^2-1}}>
  3q_{k+1}|q_{k+1}\theta-p_{k+1}|.
\]
Поскольку $1<\eta<\theta<2$ и $1\leq p_k/q_k\leq2$, $1\leq r_k/s_k\leq2$ для каждого $k$, справедливы неравенства
\begin{equation}\label{eq:bounds_for_the_plus_forms}
  2q_k<q_k+p_k\eta<5q_k,
  \qquad
  2s_k<r_k\theta+s_k<5s_k.
\end{equation}
Стало быть, для всех достаточно больших $k$ выполняется
\begin{multline}\label{eq:product_decaying_qs}
  |q_k\theta-p_k|\cdot|q_k+p_k\eta|>
  2q_k|q_k\theta-p_k|
  > \\ >
  6s_k|s_k\eta-r_k|>
  \frac65\,
  |r_k\theta+s_k|\cdot|r_k-s_k\eta|,
\end{multline}
\vskip-6mm
\begin{multline}\label{eq:product_decaying_sq}
  |r_k\theta+s_k|\cdot|r_k-s_k\eta|>
  2s_k|s_k\eta-r_k|
  > \\ >
  6q_{k+1}|q_{k+1}\theta-p_{k+1}|>
  \frac65\,
  |q_{k+1}\theta-p_{k+1}|\cdot|q_{k+1}+p_{k+1}\eta|.
\end{multline}

\paragraph{Решётка $\La$ и её относительные минимумы.}

Рассмотрим решётку $\La=A\Z^2$, где
\[
  A=
  \begin{pmatrix}
    \theta & -1 \\
    1 & \phantom{-}\eta
  \end{pmatrix}.
\]
Положим $\vec v_{-1}=(\theta,1)$, $\vec w_{-1}=(1,-\eta)$ и для каждого $k\geq0$ положим 
\[
  \vec v_k=(q_k\theta-p_k,q_k+p_k\eta)=A
  \begin{pmatrix}
    \,q_k \\
    \,p_k
  \end{pmatrix},
  \qquad
  \vec w_k=(r_k\theta+s_k,r_k-s_k\eta)=A
  \begin{pmatrix}
    \phantom{-}r_k \\
    -s_k
  \end{pmatrix}.
\]
Как сказано в пункте \ref{subsec:relative_minima_in_dimension_2}, множество относительных минимумов решётки $\La$ в точности совпадает с множеством точек
\begin{equation}\label{eq:relative_minima_of_La}
  \pm\vec v_k,\quad
  \pm\vec w_k,\qquad
  k\in\Z,\ \ k\geq-1.
\end{equation}

\paragraph{Выбор точки $\vec v_m$.}

В силу \eqref{eq:bounds_for_the_plus_forms} и \eqref{eq:consecutive_denominators_inequalities} найдётся такое $K\in\N$, что для всех $k>K$ выполняются неравенства
\begin{equation}\label{eq:supnorm_growing}
\begin{array}{l}
  |\vec v_k|<5q_k<5s_k^{1/\beta}<
  5\cdot2^{-1/\beta}|\vec w_k|^{1/\beta}<
  |\vec w_k|,\quad\,\
  \\
  \vphantom{1^{\Big|}}
  |\vec w_k|<5s_k<5q_{k+1}^{1/\beta}<
  5\cdot2^{-1/\beta}|\vec v_{k+1}|^{1/\beta}<
  |\vec v_{k+1}|,
\end{array}
\end{equation}
причём, к тому же, ввиду \eqref{eq:product_decaying_qs}, \eqref{eq:product_decaying_sq}
\begin{equation}\label{eq:Pi_decaying}
  \Pi(\vec v_k)>\frac65\,\Pi(\vec w_k),
  \qquad
  \Pi(\vec w_k)>\frac65\,\Pi(\vec v_{k+1}).
\end{equation}
Из \eqref{eq:Pi_decaying} следует, что найдётся такой индекс $m>K$, для которого
\[
  \Pi(\vec v_m)<\min_{k\leq m-1}\min\big(\Pi(\vec v_k),\Pi(\vec w_k)\big).
\]
Учитывая, что множество точек \eqref{eq:relative_minima_of_La} совпадает с множеством вершин полигонов Клейна решётки $\La$, получаем, что $\vec v_m$ является гиперболическим минимумом этой решётки.

\paragraph{Последующие гиперболические минимумы.}

Покажем, что $\vec w_m$ также является гиперболическим минимумом решётки $\La$. Рассмотрим прямоугольники
\[
\begin{array}{ll}
  \cP_1=\Big\{(z_1,z_2)\in\R^2 \,\Big|\, |z_1|\leq|q_m\theta-p_m|,\ |z_2|\leq q_{m+1}+p_{m+1}\eta \Big\},
  \\ \vphantom{\Bigg|}
  \cP_2=\Big\{(z_1,z_2)\in\R^2 \,\Big|\, |z_1|\leq|q_m\theta-p_m|,\ |z_2|\leq r_m\theta+s_m \Big\},
  \\
  \cP_3=\Big\{(z_1,z_2)\in\R^2 \,\Big|\, |z_1|\leq r_m\theta+s_m,\ |z_2|\leq|q_m\theta-p_m| \Big\}
\end{array}
\]
и параллелограммы
\[
\begin{array}{ll}
  \cQ_1=\Big\{(x,y)\in\R^2 \,\Big|\, |x\theta-y|\leq|q_m\theta-p_m|,\ |x|\leq q_{m+1} \Big\},
  \\ \vphantom{\Bigg|}
  \cQ_2=\Big\{(x,y)\in\R^2 \,\Big|\, |y|\leq s_m,\ |x+y\eta|\leq|q_m\theta-p_m| \Big\},
  \\
  \cQ_3=\Big\{(x,y)\in\R^2 \,\Big|\, |y|\leq s_m,\ |x+y\eta|\leq|s_{m-1}\eta-r_{m-1}| \Big\}.
  \quad\ \
\end{array}
\]

\begin{figure}[h]
\centering
\begin{tikzpicture}[domain=-5.1:5.1,scale=1.3]

  \draw[->,>=stealth'] (-5.5,0) -- (5.5,0) node[right] {$z_2$};
  \draw[->,>=stealth'] (0,-5.5) -- (0,5.5) node[above] {$z_1$};

  \fill[fill=blue,opacity=0.15]
      plot [domain=0.2:5] (\x,{1/(\x)}) --
      plot [domain=5:0.2] (\x, {-1/(\x)}) --
      plot [domain=-0.2:-5] (\x,{1/(\x)}) --
      plot [domain=-5:-0.2] (\x, {-1/(\x)}) --
      cycle;

  \fill[fill=red,opacity=0.15]
      plot [domain=0.75:2] (\x,{1.5/(\x)}) --
      plot [domain=2:0.75] (\x, {-1.5/(\x)}) --
      plot [domain=-0.75:-2] (\x,{1.5/(\x)}) --
      plot [domain=-2:-0.75] (\x, {-1.5/(\x)}) --
      cycle;

  \fill[fill=teal,opacity=0.1]
      (-5,0.75) -- (5,0.75) -- (5,-0.75) -- (-5,-0.75) -- cycle;

  \fill[fill=teal,opacity=0.1]
      (0.75,-5) -- (0.75,5) -- (-0.75,5) -- (-0.75,-5) -- cycle;

  \draw[color=blue]
      plot [domain=0.2:5] (\x,{1/(\x)}) --
      plot [domain=5:0.2] (\x, {-1/(\x)}) --
      plot [domain=-0.2:-5] (\x,{1/(\x)}) --
      plot [domain=-5:-0.2] (\x, {-1/(\x)}) --
      cycle;

  \draw[color=red]
      plot [domain=0.75:2] (\x,{1.5/(\x)}) --
      plot [domain=2:0.75] (\x, {-1.5/(\x)}) --
      plot [domain=-0.75:-2] (\x,{1.5/(\x)}) --
      plot [domain=-2:-0.75] (\x, {-1.5/(\x)}) --
      cycle;
  
  \draw[color=teal] (-5,0.75) -- (5,0.75) -- (5,-0.75) -- (-5,-0.75) -- cycle;

  \draw[color=teal] (0.75,-5) -- (0.75,5) -- (-0.75,5) -- (-0.75,-5) -- cycle;

  \node[fill=black,circle,inner sep=1.3pt] at (0,0) {};
  \node[fill=black,circle,inner sep=1.3pt] at (2,0.75) {};
  \node[fill=black,circle,inner sep=1.3pt] at (-2,-0.75) {};
  \node[fill=black,circle,inner sep=1.3pt] at (-1/5,5) {};
  \node[fill=black,circle,inner sep=1.3pt] at (1/5,-5) {};
  
  \draw(1.95,0.9) node[right]{$\vec v_m$};
  \draw(-1.9,-0.9) node[left]{$-\vec v_m$};
  \draw(-0.18,5.15) node[left]{$\vec w_m$};
  \draw(0.02,-5.17) node[right]{$-\vec w_m$};

  \draw(4,0.22) node[above,color=blue]{$\cH(\vec w_m)$};
  \draw(1.27,-1.35) node[right,color=red]{$\cH(\vec v_m)$};
  \draw(-4,0.75) node[above,color=teal]{$\cP_2$};
  \draw(0.75,4) node[right,color=teal]{$\cP_3$};

\end{tikzpicture}
\caption{Покрытие $\cH(\vec w_n)$}\label{fig:H_w_m_covering}
\end{figure}

Поскольку $p_m/q_m$ и $p_{m+1}/q_{m+1}$ --- последовательные подходящие дроби числа $\theta$,
\begin{equation}\label{eq:points_in_Q1}
  \cQ_1\cap\Z^2=\big\{\vec 0,\pm(p_m,q_m),\pm(p_{m+1},q_{m+1})\big\}.
\end{equation}
Из неравенств $1<\eta<\theta<2$ легко вывести, что для любого $\vec z\in\R^2$ справедливо
\begin{equation}\label{eq:bounds_for_A}
  |\vec z|_2<|A(\vec z)|_2<5|\vec z|_2,
\end{equation}
где $|\cdot|_2$ обозначает евклидову норму. Стало быть, при достаточно большом $K$ имеет место включение
\[
  \cP_1\subset A(\cQ_1)\cup(\vec v_{m+1}+\cB)\cup(-\vec v_{m+1}+\cB),
\]
где $\cB$ обозначает открытый круг радиуса $1$ с центром в начале координат. В этом круге ненулевых точек решётки $\La$ нет. Следовательно, в силу \eqref{eq:points_in_Q1} имеем
\[
  \cP_1\cap\La=\big\{\vec 0,\pm\vec v_m,\pm\vec v_{m+1}\big\}.
\]
Из \eqref{eq:supnorm_growing} следует, что $q_m+p_m\eta < r_m\theta+s_m < q_{m+1}+p_{m+1}\eta$, то есть $\cP_2\subset\cP_1$ и
\begin{equation}\label{eq:points_in_P2}
  \cP_2\cap\Z^2=\big\{\vec 0,\pm\vec v_m\big\}.
\end{equation}

Далее, поскольку $r_{m-1}/s_{m-1}$ и $r_m/s_m$ --- последовательные подходящие дроби числа $\eta$,
\[
  \cQ_3\cap\Z^2=\big\{\vec 0,\pm(r_{m-1},-s_{m-1}),\pm(r_m,-s_m)\big\}.
\]
Из \eqref{eq:product_decaying_qs}, \eqref{eq:product_decaying_sq} и \eqref{eq:supnorm_growing} следует, что $|s_m\eta-r_m|<|q_m\theta-p_m|<|s_{m-1}\eta-r_{m-1}|$, то есть $\cQ_2\subset\cQ_3$ и
\begin{equation}\label{eq:points_in_Q2}
  \cQ_2\cap\Z^2=\big\{\vec 0,\pm(r_m,-s_m)\big\}.
\end{equation}
Поскольку для каждого $\vec z\in\R^2$ справедливо \eqref{eq:bounds_for_A}, при достаточно большом $K$ имеет место включение
\[
  \cP_3\subset A(\cQ_2)\cup(\vec w_m+\cB)\cup(-\vec w_m+\cB).
\]
В круге $\cB$ ненулевых точек решётки $\La$ нет, стало быть, в силу \eqref{eq:points_in_Q2} имеем
\begin{equation}\label{eq:points_in_P3}
  \cP_3\cap\La=\big\{\vec 0,\pm\vec w_m\big\}.
\end{equation}

Наконец, поскольку $\Pi(\vec w_m)<\Pi(\vec v_m)$, справедливо включение
\[
  \cH(\vec w_m)
  \subset
  \cH(\vec v_m)\cup\cP_2\cup\cP_3,
\]
откуда, учитывая \eqref{eq:points_in_P2}, \eqref{eq:points_in_P3} и тот факт, что $\vec v_m$ является гиперболическим минимумом решётки $\La$, получаем, что и $\vec w_m$ является гиперболическим минимумом этой решётки.

Аналогичным образом доказывается, что и все последующие точки $\vec v_k$, $\vec w_k$ при $k\geq m+1$ являются гиперболическими минимумами решётки $\La$.

\paragraph{Значение экспоненты $\hat{\hat\omega}(\Lambda)$.}

Ввиду \eqref{eq:bounds_for_the_products} и \eqref{eq:denominators_asymptotics} справедливы асимптотические оценки
\begin{multline*}
  \Pi(\vec v_k)=
  \big(|q_k\theta-p_k|\cdot|q_k+p_k\eta|\big)^{1/2}
  \asymp
  \big(q_k|q_k\theta-p_k|\big)^{1/2}
  \asymp \\ \asymp
  q_k^{(1-\beta^2)/2}
  \asymp
  s_k^{(\beta^{-1}-\beta)/2}
  \asymp
  (r_k\theta+s_k)^{(\beta^{-1}-\beta)/2}=
  |\vec w_k|^{(\beta^{-1}-\beta)/2},
\end{multline*}
\vskip-8mm
\begin{multline*}
  \Pi(\vec w_k)=
  \big(|r_k\theta+s_k|\cdot|r_k-s_k\eta|\big)^{1/2}
  \asymp
  \big(s_k|s_k\eta-r_k|\big)^{1/2}
  \asymp \\ \asymp
  s_k^{(1-\beta^2)/2}
  \asymp
  q_{k+1}^{(\beta^{-1}-\beta)/2}
  \asymp
  (q_{k+1}+p_{k+1}\eta)^{(\beta^{-1}-\beta)/2}=
  |\vec v_{k+1}|^{(\beta^{-1}-\beta)/2}.
\end{multline*}
Стало быть, если $k$ достаточно велико и $t_1,t_2$ удовлетворяют неравенствам
\[
  q_k+p_k\eta\leq t_1< 
  r_k\theta+s_k\leq t_2< 
  q_{k+1}+p_{k+1}\eta,
\]
то
\[
\begin{array}{ll}
  |\vec v_k|=
  q_k+p_k\eta
  \leq t_1,
  &\quad
  \Pi(\vec v_k)
  \asymp
  (r_k\theta+s_k)^{(\beta^{-1}-\beta)/2}<
  t_1^{(\beta^{-1}-\beta)/2},
  \\
  |\vec w_k|=
  r_k\theta+s_k
  \leq t_2,
  &\quad
  \Pi(\vec w_k)
  \asymp
  (q_{k+1}+p_{k+1}\eta)^{(\beta^{-1}-\beta)/2}<
  t_2^{(\beta^{-1}-\beta)/2}.
  \vphantom{1^{\Big|}}
\end{array}
\]
Отсюда ввиду \eqref{eq:hat_hat_omega_in_terms_of_hyperbolic_minima} следует, что
\[
  \hat{\hat\omega}(\Lambda)\geq(\beta-\beta^{-1})/2.
\]
Поскольку же $\Pi(\vec v_k)\asymp|\vec w_k|^{(\beta^{-1}-\beta)/2}$ и во внутренности множества
\[
  \Big\{\vec z\in\R^d \,\Big|\, |\vec z|\leq|\vec w_k|,\ \Pi(\vec z)\leq\Pi(\vec v_k) \Big\}
\]
нет ненулевых точек решётки, приходим к равенству
\[
  \hat{\hat\omega}(\Lambda)=(\beta-\beta^{-1})/2.
\]

\paragraph{Описание спектра.}

Образом луча $(1,\infty)$ при отображении
\[
  \beta\mapsto(\beta-\beta^{-1})/2
\]
является луч $(0,\infty)$. Стало быть, все положительные числа принадлежат спектру значений $\hat{\hat\omega}(\Lambda)$.

Далее, если в качестве $\theta$ и $\eta$ взять иррациональные числа с ограниченными неполными частными, то для таких чисел $\omega(\theta)=\omega(\eta)=1$ и в соответствии с \eqref{eq:omega_La_vs_the_maximum} имеем $\omega(\La)=0$. Поскольку же $0\leq\hat{\hat\omega}(\La)\leq\omega(\La)$, получаем $\hat{\hat\omega}(\La)=0$.

Значение $\infty$ достигается, например, на тех решётках, содержащих ненулевые точки, среди координат которых есть ноль. Но можно построить и решётки без таких точек, на которых также достигается значение $\infty$. Например, если вместо роста неполных частных, задаваемого \eqref{eq:next_partial_qoutients_definition}, рассмотреть более быстрый рост --- к примеру, определить неполные частные соотношениями
\[
  a_k=\bigg[\frac{s_{k-1}^k-q_{k-2}}{q_{k-1}}\bigg]+1,
  \qquad
  b_k=\bigg[\frac{q_k^k-s_{k-2}}{s_{k-1}}\bigg]+1,
\]
то, следуя шагам рассуждения, приведённого выше, несложно показать, что соответствующие $\theta$ и $\eta$ дадут решётку $\La$, для которой будет выполнено $\hat{\hat\omega}(\La)=\infty$.

Теорема доказана.

\section*{Благодарности}

Автор 
благодарен профессору Яну Бюжо за ценные замечания о природе экспоненты $\hat{\hat\omega}(\La)$.

\end{document}